\documentclass[11pt]{article}
\usepackage{amsthm,amssymb,mathtools,mathrsfs,xcolor}

\newcommand{\R}{\mathbb{R}}

\newcommand{\eps}{\varepsilon}

\theoremstyle{definition}
\newtheorem{definition}{Definition}[section] 
\theoremstyle{plain}
\newtheorem{theorem}[definition]{Theorem}
\newtheorem{lemma}[definition]{Lemma}

\theoremstyle{remark}
\newtheorem{remark}[definition]{Remark}

\newcommand{\diam}{\operatorname{diam}}

\def\Xint#1{\mathchoice
{\XXint\displaystyle\textstyle{#1}}%
{\XXint\textstyle\scriptstyle{#1}}%
{\XXint\scriptstyle\scriptscriptstyle{#1}}%
{\XXint\scriptscriptstyle\scriptscriptstyle{#1}}%
\!\int}
\def\XXint#1#2#3{{\setbox0=\hbox{$#1{#2#3}{\int}$ }
\vcenter{\hbox{$#2#3$ }}\kern-.57\wd0}}

\def\dashint{\Xint-}

\begin{document}

\title{\bf{Accessible parts of the boundary for domains in metric measure spaces}}
\date{}
\author{Ryan Gibara and Riikka Korte}
\maketitle

\begin{abstract}

In English:
We prove in the setting of $Q$--Ahlfors regular PI--spaces the following result: if a domain has uniformly large boundary when measured with respect to the $s$--dimensional Hausdorff content, then its visible boundary has large $t$--dimensional Hausdorff content for every $0<t<s\leq Q-1$. The visible boundary is the set of points that can be reached by a John curve from a fixed point $z_{0}\in \Omega$. This generalizes recent results by Koskela-Nandi-Nicolau (from $\mathbb R^2$) and Azzam ($\mathbb R^n$). In particular, our approach shows that the phenomenon is independent of the linear structure of the space.

In Finnish:
Title: Alueen näkyvä reuna metrisissä avaruuksissa
Abstract:
 Osoitamme Ahlfors-säännöllisissä metrisissä avaruuksissa seuraavan tuloksen: Jos alueen reuna on tasaisesti suuri $s$-uloitteisesen Hausdorffin mitan suhteen, tällöin sen näkyvä reuna on suuri $t$-uloitteisen Hausdorffin mitan suhteen kaikilla $0<t<s\leq Q-1$. Näkyvällä reunalla tarkoitamme niitä pisteitä, jotka voidaan saavuttaa John-poluilla jostain kiinnitetystä pisteestä. Tuloksemme yleistää  Koskelan, Nandin ja Nicolaun ($\mathbb R^2$) sekä Azzamin ($\mathbb R^n$) tuoreita tuloksia. Erityisesti konstruktiivinen menetelmämme osoittaa, että tämä ilmiö ei ole riippuvainen avaruuden lineaarisesta rakenteesta.

\end{abstract}



\section{Introduction}

We say that a domain (that is, a connected open set) $\Omega\subset\R^n$ is $c$--John with centre $z_0\in\Omega$ and constant $c\geq 1$ if every $z\in\Omega$ can be joined to $z_0$ by a $c$--John path. That is, there exists 
a path $\gamma$ such that 
\[
\ell\big(\gamma(z',z) \big)\leq c\,d_{\Omega}(z')
\]
for all $z'$ in the image of $\gamma$, where $\ell\big(\gamma(z',z) \big)$ is the length of the subpath joining $z'$ to $z$ and $d_{\Omega}(z')=d(z',\Omega^{c})$. For such domains, every point $\omega\in\partial\Omega$ is accessible in a non-tangential sense: there exists a $c$--John path connecting the centre $z_0$ to $\omega$. 

Given a domain $\Omega$, not necessarily John, fix a point $z_0\in\Omega$ and a constant $c\geq{1}$. We consider the largest $c$--John subdomain of $\Omega$ with centre $z_0$, $\Omega_{z_0}(c)$, and call the set $\partial\Omega_{z_0}(c)\cap\partial\Omega$ the $c$--accessible (or $c$--visible) boundary of $\Omega$ near $z_0$. This corresponds to those points on $\partial\Omega$ that can be reached from $z_0$ by a $c$--John path. 

One can ask if the visible boundary of a domain is large in the sense that there exist $C>0$ and $0\leq t\leq n$ such that
\begin{equation}
\label{visibility}
\mathcal{H}^{t}_{\infty}( \partial\Omega_{z_0}(c) \cap \partial \Omega  ) \geq C d_{\Omega}(z_0)^{t}
\end{equation}
for some $c\geq{1}$ and all $z_0\in\Omega$. Here, the $t$--dimensional Hausdorff content of a subset $E\subset\R^n$ is defined as
$$
\mathcal{H}^{t}_{\infty}(E)=\inf\left\{\sum_{i=1}^{\infty}r_i^{t}:E\subset\bigcup_{i=1}^{\infty}B(x_i, r_i) \right\},
$$
where $B(x_i, r_i)$ denotes a ball (open or closed) centred at $x_i$ of radius $r_i$. The case $t=0$ is not interesting as condition \eqref{visibility} is trivially satisfied with $c=1$ for any proper subdomain of $\R^n$. On the other hand, condition \eqref{visibility} can fail for $t=n-1$ even when $\Omega$ is assumed to satisfy some nice geometric properties (see \cite{azzam}), and so the focus is on the interval $0<t<n-1$.

As pointed out by Koskela and Lehrb\"{a}ck in \cite{kl}, if a uniform domain $\Omega$ has uniformly large Hausdorff content in the sense that, for all $z\in \Omega$, 
$$
\mathcal{H}_{\infty}^{t}\big( B(z,2 d_{\Omega}(z)) \cap \partial\Omega \big) \geq C_0 d_{\Omega}(z)^{t} 
$$
for some $C_0>0$ and $0<t\leq n$, then it has large visible boundary; i.e. $\Omega$ satisfies \eqref{visibility} with the same $t$ and some $c\geq{1}$. A domain is called uniform if there is a constant $c_{u}\geq{1}$ such that all pairs $z_0,z_1\in \Omega$ can be joined by a path $\gamma$ such that $\ell(\gamma)\leq c_{u}|z_0-z_1|$ and $d_{\Omega}(z)\geq c_{u}^{-1}\min\{|z-z_0|,|z-z_1|\}$ for all $z$ in the image of $\gamma$.

A first result without assuming uniformity of the domain is proven by Koskela, Nandi, and Nicolau in \cite{knn} where they use techniques from complex analysis to show that any bounded simply connected domain in the complex plane satisfies $\eqref{visibility}$ for all $0<t<1$.

Shortly after, the result was generalized to $\mathbb R^n$. More precisely, Azzam proved in \cite{azzam} that if for some $0< s\leq n-1$ and $C_0>0$,
$$
\mathcal{H}_{\infty}^{s}\big( B(\omega,\lambda) \setminus \Omega \big) \geq C_0 \lambda^{s}
$$ 
holds for all $\omega\in\partial\Omega$ and $0<\lambda<\diam(\Omega)$, a condition he calls having lower $s$--content regular complement, then \eqref{visibility} holds for some $c\geq{1}$ and for $0 <t<s$. In fact, Azzam shows the stronger result with the visible boundary defined with chord-arc subdomains playing the role of John subdomains. The basis of his proof is the construction of a subset of the visible boundary that makes use of projections, and so it features strong reliance on the linear structure of the space.

The main result of the present paper is the following, a new proof showing that the phenomenon holds in the nonlinear setting. Our proof is based on a very flexible construction of a path family that is natural in more general metric measure spaces than $\mathbb{R}^n$. On the other hand, our estimates are based on iterative arguments, and so we do not obtain sharp dimensions if the measure is merely doubling, see Remark \ref{remark}.

\begin{theorem}\label{maintheorem}
Let $(X,d,\mu)$ be a complete $Q$--Ahlfors regular metric measure space supporting a weak $(1,p)$--Poincar\'{e} inequality, $1\leq p <\infty$. Fix $0<s\leq Q-1$. Let $\Omega\subset X$ be a domain such that for all $\omega\in\partial\Omega$ and all $0<\lambda<\diam(\Omega)$ we have
\begin{equation}\label{thickboundary}
\mathcal{H}_{\infty}^{s}\big( B(\omega,\lambda) \cap \partial\Omega \big)\geq {C_0} \lambda^{s} 
\end{equation}
for some $C_0>0$. Then for all $0<\eps<s$ there exist $c\geq 1$ and $C>0$ such that
\begin{equation}\label{bigvisible}
\mathcal{H}^{s-\eps}_{\infty}( \partial\Omega_{z_0}(c) \cap \partial \Omega  ) \geq C d_{\Omega}(z_0)^{s-\eps}
\end{equation}
for all $z_0\in\Omega$.
\end{theorem}

The study of the size of the visible boundary is partially motivated by its relationship to Hardy inequalities. It was shown in \cite{kl}, in the Euclidean setting, that if a domain has large visible boundary in the sense of \eqref{bigvisible}, then it admits certain Hardy inequalities. This was generalised to metric measure spaces in \cite{lehrback}, where the author also shows that it is enough for a domain to satisfy \eqref{thickboundary} to guarantee that it admits a Hardy inequality. The present work complements these results by showing, in fact, that \eqref{thickboundary} implies \eqref{bigvisible}.

We begin with some preliminaries, reviewing the definitions and notions relevant to the analysis on metric measure spaces. In Section 3, we construct a set $P_\infty$ that, in Section 4, is shown to be a subset of the visible boundary $\partial\Omega_{z_0}(c)\cap \partial \Omega$ for some $c\geq{1}$ and to satisfy $\mathcal{H}^{s-\eps}_{\infty}( P_\infty  ) \geq C d_{\Omega}(z_0)^{s-\eps}$, implying \eqref{bigvisible}. In Section 5, an example is given showing that it is necessary for \eqref{thickboundary} to hold at all $\omega$ and at all scales $\lambda$.

\section{Preliminaries}

Let $(X,d,\mu)$ be a metric measure space. By this we mean that $(X,d)$ is a metric space endowed with a non-trivial Borel regular (outer) measure $\mu$ such that $0<\mu(B)<\infty$ for all balls $B\subset{X}$ of positive and finite radius. 
The notations $B(r)$ and $B(x,r)$ will be used when the radius, or the centre and radius, respectively, must be specified.

Throughout the rest of the paper, we assume that $(X,d,\mu)$ is a complete $Q$--Ahlfors regular metric measure space supporting a weak $(1,p)$--Poincar\'{e} inequality for sufficiently small exponent $p$. It turns out that, without loss of generality, it is enough to prove the theorem under some additional assumptions -- see Remark \ref{geodesic}. We review these definitions below for the reader's benefit. A class of metric measure spaces satisfying all of these assumptions is the Carnot-Carath\'{e}odory spaces \cite{amp}, an important example of which is the Heisenberg group.

\begin{definition}
We say that the measure $\mu$ is $Q$--Ahlfors regular, $1<Q<\infty$, if there exists a constant $c_A>0$ such that for all balls $B(r)\subset{X}$, 
\[
\frac{1}{c_{A}}r^{Q}\leq\mu\big(B(r)\big)\leq c_{A}\,r^{Q}.
\]
\end{definition}
A measure being $Q$--Ahlfors regular implies that it is doubling; i.e.
\[
\mu\big(B(2r)\big)\leq c_d\mu\big(B(r)\big)
\]
with $c_{d}=2^Qc_{A}^2$. The optimal doubling constant may be smaller, however. Note that a complete and doubling metric measure space is proper: bounded and closed sets are compact.

Let $f$ be a locally Lipschitz function. Then its local Lipschitz constant is defined as
$$
\mathrm{Lip}f(x)=\liminf_{r\rightarrow{0}}\sup_{y\in B(x,r)}\frac{|f(x)-f(y)|}{d(x,y)}.
$$

\begin{definition}
We say that $X$ supports a weak $(1,p)$--Poincar\'{e} inequality, $1\leq p <\infty$, if there exist constants $c_P>0$ and $\tau\geq 1$ such that for all balls $B(r)\subset{X}$ and all locally Lipschitz functions $f$, we have 
$$
\dashint_{B(r)}\!|f-f_{B(r)}|\,\mathrm{d}\mu\leq c_Pr\left(\dashint_{B(\tau r)}\!(\mathrm{Lip}f)^p\,{d}\mu \right)^{1/p}.
$$
Here, and elsewhere, for any set $E\subset{X}$ with $0<\mu(E)<\infty$, we write
$$
f_{E}=\dashint_{E}\!{f}\,{d}\mu=\frac{1}{\mu(E)}\int_{E}\!{f}\,{d}\mu.
$$
\end{definition}

Throughout the paper, we will often be measuring the size of sets in terms of its Hausdorff content.

\begin{definition} 
Let $E\subset{X}$. The $\alpha$--dimensional Hausdorff content of $E$ is defined as 
$$
\mathcal{H}^{\alpha}_{\infty}(E)=\inf\left\{\sum_{i=1}^{\infty}r_{i}^{\alpha}\,:\,E\subset\bigcup_{i=1}^{\infty}B(x_i, r_i) \right\}.
$$
\end{definition}

Next we define the concepts mentioned in the introduction within the setting of metric measure spaces. For a set $E\subset{X}$ we write $d_E(z)=d(z,X\setminus E)$ for $z\in E$. By a path we mean a rectifiable, nonconstant compact curve $\gamma$ on $X$. Such curves can always be parametrized by arclength (see \cite{buse}); as such, all paths will be assumed to have domain $[0,\ell(\gamma)]$, where $\ell(\gamma)$ is the arclength of $\gamma$.  
We denote by $i(\gamma)$ the image of $\gamma$ in $X$ and by $\gamma(x,y)$ the subpath of $\gamma$ joining $x,y\in i(\gamma)$. For a set $E\subset X$ and a path $\gamma$ in $X$, we write
$$
d_E(\gamma)=\inf_{z\in i(\gamma)}d_E(z). 
$$

\begin{definition}
A path $\gamma:[0,\ell]\rightarrow{\Omega}$ with $x=\gamma(0)$ and $y=\gamma(\ell(\gamma))$ is called $c$--John if there exists a constant $c\geq{1}$ such that 
\[\ell\big(\gamma(z,y) \big)\leq c\,d_{\Omega}(z)
\] 
for all $z\in i(\gamma)$. A domain $\Omega\subset X$ is called $c$--John with centre $z_0\in\Omega$ if for every $z\in\Omega$ there exists a $c$--John path $\gamma:[0,\ell(\gamma)]\rightarrow{\Omega}$ with $z_0=\gamma(0)$ and $z=\gamma(\ell(\gamma))$. 
\end{definition}

If $\Omega\subset X$ is a $c$--John domain, then for every $\omega\in\partial\Omega$ (or, equivalently, for every $\omega \in \overline\Omega$) there exists a path $\gamma:[0,s]\rightarrow{\Omega\cup\{\omega\}}$ with $z_0=\gamma(0)$ and $\omega=\gamma(s)$ such that $\ell\big(\gamma(z,\omega) \big)\leq cd_{\Omega}(z)$ for all $z\in i(\gamma)$.

Note that a $c$--John domain $\Omega$ with centre $z_0$ satisfies $\Omega\subset B\big(z_0,cd_{\Omega}(z_0)\big)$. In particular, it is bounded.

\begin{definition}
Given a domain $\Omega\subset X$, a point $z_0\in\Omega$, and a constant $c\geq{1}$, write
$$
\Omega_{z_0}(c)=\bigcup\{U\subset\Omega: U\,\text{is a $c$--John domain with centre $z_0$} \}.
$$
The set $\partial\Omega_{z_0}(c)\cap\partial\Omega$ is called the visible, or accessible, boundary of $\Omega$ near $z_0$.
\end{definition}

\begin{remark}\label{geodesic}
Recall that our metric measure space is always assumed to be complete and Ahflors regular, and to support a weak Poincar\'{e} inequality. In this setting, we may assume without loss of generality that the metric is geodesic. Specifically, $d$ is quasiconvex, thus bi-Lipschitz equivalent to a geodesic metric $d'$ \cite[Theorem 8.3.2]{hkst}. The space retains all of its properties under this change of metric, but with modified constants. In fact, under this geodesic metric, we may assume that $\tau=1$ in the Poincar\'{e} inequality \cite[Theorem 4.18]{hei}. Additionally, a path is John with respect to $d$ if and only if it is John with respect to $d'$, but again with a potentially different constant. Therefore, we will assume throughout the paper that we have selected the geodesic metric and that $\tau=1$. 
\end{remark}

A consequence of selecting the geodesic metric is that balls are John domains, from which it follows that $\Omega_{z_0}(c)\neq\emptyset$ by the openness of $\Omega$. Then, it follows from the definition of visible boundary that $\Omega_{z_0}(c)$ is a John domain inside $\Omega$.

Some monographs on metric measure spaces where the reader may learn about these and further topics in the field are \cite{bb}, \cite{hei}, and \cite{hkst}.

\section{The Construction}
\label{section:construction}
In this section, we provide an iterative construction of a subset $P_{\infty}$ of the boundary of $\Omega$. In Section \ref{sec:proof}, we will prove that $P_{\infty}$ is part of the visible boundary and give a lower bound for its Hausdorff content.

To make the description of the construction simpler, we introduce some more definitions.

\begin{definition}
A collection of balls $\{B_i=B(x_{i},r) \}_{i\in I}$ for some common radius $r>0$ is said to be finitely chainable if for all $i,j\in I$ there exists a finite subcollection of balls $\{B_{i_1},B_{i_2},\ldots,B_{i_m} \}$ such that $i_1=i$, $i_m=j$, and $x_{i_{k+1}}\in B_{i_k}$ for all $k=1,2,\ldots,m-1$.
\end{definition}

In the construction of the visible boundary, we need to find plenty of balls that touch the boundary of $\Omega$ at all scales. In addition, the balls need to be far enough from each other so that the ``descendants'' of the balls will not touch the balls in different branches. These conditions will be fulfilled by balls that satisfy the next definition.

\begin{definition}
A collection of balls $\{B_i \}_{i\in I}$ 
 is said to be well placed along some set $F$ 
 if $\partial{B_i}\cap F\neq\emptyset$ for all $i\in I$ and $4B_{i}\cap 4B_{j}=\emptyset$ for every $i,j\in{I}$, $i\neq{j}$.
\end{definition}

Now we commence the construction. Fix $z_0\in\Omega$ and $0<\eta<1$, and set $r=r_{0}=d_{\Omega}(z_0)$. The value of $\eta$ will be specified later.

\underline{\bf Step 0}: Consider the ball $B(z_0,r)$ and choose a point $\omega_0\in\partial B(z_0,r)\cap\partial\Omega$. Write $P_0=\{\omega_0\}$.

\underline{\bf Step 1}: Consider the ball $B_1=B(\omega_0,2r)$. Denote by $\tilde{\mathcal{B}}_1$ a maximal finitely chainable collection of balls in $B_1\cap\Omega$ of radius $\eta r$ such that $B(z_0,\eta r)\in\tilde{\mathcal{B}}_1$. 
Then, consider $\mathcal{B}_1$, a maximal subcollection of $\tilde{\mathcal{B}}_1$ that is well placed along $\partial\Omega$.
For each ball $B\in \mathcal{B}_1$, consider a point $\omega\in\partial{B}\cap\partial{\Omega}$ and write $P_1$ for the collection of all such points.

\underline{\bf Step $k+1$}: Fix an $\omega\in P_{k}$ (that is $\omega\in \partial B(z,\eta^{k} r)\cap\partial\Omega$ for some $z\in\Omega$) and consider the ball $B_{k+1}=B(\omega,2\eta^{k}r)$. Denote by $\tilde{\mathcal{B}}_{k+1}(\omega)$ a maximal finitely chainable collection of balls in $B_{k+1}\cap\Omega$ of radius $\eta^{k+1}r$ such that $B(z,\eta^{k+1}r)\in\tilde{\mathcal{B}}_{k+1}(\omega)$.
Then, consider $\mathcal{B}_{k+1}(\omega)$, a maximal subcollection of $\tilde{\mathcal{B}}_{k+1}(\omega)$ that is well placed along $\partial\Omega$. Write 
$$
\mathcal{B}_{k+1}=\bigcup_{\omega\in P_k}\mathcal{B}_{k+1}(\omega).
$$
For each ball $B\in \mathcal{B}_{k+1}$, consider a point $\omega\in\partial{B}\cap\partial{\Omega}$ and write $P_{k+1}$ for the collection of all such points.

The set $P_\infty\subset B(\omega_0,2r)\cap\partial{\Omega}$ is then defined as $P_\infty=\overline{\bigcup_{k\geq{1}}P_k}$. 

\begin{remark}\label{remarknested}
Due to the assumed geodecity of the metric (see Remark \ref{geodesic}), this construction can be done in such a way that each each point $w\in P_k$ is also in $P_j$ for all $j>k$. Indeed, if $\omega \in \partial B \cap \Omega$ with $B\in \mathcal B_k$, then $B(\tilde\omega,\eta^{k+1}r)\in \tilde{\mathcal B}_{k+1}$, where $\tilde \omega$ is the point along the geodesic connecting $\omega$ to the centre of $B$ with distance $\eta^{k+1}r$ from the boundary of $B$. Geodecity of the space guarantees that $B(\tilde\omega,\eta^{k+1}r)\subset B$ and that $\omega$ is on the boundary of $B(\tilde\omega,\eta^{k+1}r)$. If the space were not geodesic, there might not exist a ball of radius $\eta^{k+1}r$ with $\omega$ on its boundary that is contained in $\Omega$. This will be useful in later lemmas.
\end{remark}

\section{Proof of Result}\label{sec:proof}

To complete the proof of Theorem~\ref{maintheorem}, we show that the set $P_{\infty}$ that was constructed in the previous section is part of the visible boundary and then we give an estimate for the Hausdorff content of $P_{\infty}$.
We start with two simple technical lemmas, Lemma \ref{lem:2b} and Lemma \ref{lem:4John}, that are needed for proving that $P_{\infty}$ is part of the visible boundary.

\begin{lemma}\label{lem:2b}
Let $\{B(x_i,r) \}_{i=1}^S $ be a finitely chainable collection of balls of some fixed radius. If $x$ and $y$ are centres of balls from $\{B(x_i,r) \}_{i=1}^S $, then there exists a path $\gamma_{x,y}$ such that 
\[
d_{{E}}(\gamma_{x,y})\geq\frac{r}{2}\qquad \text{and} \qquad \ell(\gamma_{x,y})\leq S r,
\]
where $E=\cup_{i=1}^{S} B(x_{i},r)$.
\end{lemma}
\begin{proof} As the collection is finitely chainable, there exists a sequence of points $x_{i_{0}}=x,x_{i_{1}},\ldots, x_{i_{k}}=y$ with $k<S$ such that $x_{i_{j}}\in B(x_{i_{j-1}},r)$ for each $j=1,\ldots, k$. Let $\gamma_{j}$ be a geodesic connecting $x_{i_j}$ to $x_{i_{j+1}}$. Then $\ell(\gamma_{j})\leq r$ and $i(\gamma_{j})\subset B(x_{i_{j}},r/2)\cup B(x_{i_{j+1}},r/2)$. Consequently, 
\[
d_{E}(\gamma_{j})\geq d_{B(x_{i_{j}},r)\cup B(x_{i_{j+1}},r) }(\gamma_{j})\geq r/2.
\]
Thus $\gamma=\gamma_{1}\cup\cdots\cup \gamma_{k}$ is a path that satisfies the required conditions.
\end{proof}

\begin{lemma}
\label{lem:4John}
Let $M>{1}$ and $0<\eta<1$. Let $\{\gamma_k\}$ be a sequence, finite or infinite, of paths with images in some ball $B\subset\Omega$ of radius $r>0$ such that, for each $k$, $\gamma_{k}(\ell(\gamma_{k}))=\gamma_{k+1}(0)$, $\ell(\gamma_k)\leq M\eta^k r$ and $d_{\Omega}(\gamma_k)>\frac{1}{M}\eta^k r$. Then $\gamma=\gamma_1\cup\gamma_2\cup\cdots$ is a $c$--John path with $c=c(\eta,M)$.
\end{lemma}

\begin{proof}
Denote by $z_{k-1}$ and $z_k$ the initial and terminal points, respectively, of $\gamma_k$. First, we show that $\gamma_k$ is $M^2$--John for each $k$. Fixing $k$ and $z\in i(\gamma_k)$, we have 
$$
\ell\big(\gamma_k(z,z_k) \big)\leq \ell(\gamma_k)\leq M\eta^k r = M^2\left(\frac{1}{M}\eta^k r \right)<M^2d_{\Omega}(\gamma_k)\leq M^2d_{\Omega}(z).
$$

Begin by assuming that $\gamma$ is comprised of $N$ paths. We show that $\gamma=\gamma_1\cup\gamma_2\cup\ldots\cup \gamma_{N}$ is $\tfrac{M^2}{1-\eta}$-John. If $z\in i(\gamma_k)$, then 
we have that
\begin{eqnarray*}
\ell\big(\gamma(z,z_N) \big)&=&\ell\big(\gamma_k(z,z_k) \big)+\ell(\gamma_{k+1})+\ldots+\ell(\gamma_N)\\
&\leq& M(\eta^{k}+\ldots+\eta^{N})r\\
&<& M(\eta^{k}+\ldots+\eta^{N})M\eta^{-k}d_{\Omega}(\gamma_k)\\
&\leq & \frac{M^{2}}{1-\eta}d_{\Omega}(z).
\end{eqnarray*}

Thus $\gamma$ is $\frac{M^2}{1-\eta}$--John when $\gamma$ is formed by a finite union. As the estimate does not depend on $N$, we can pass to the limit and the result holds also for an infinite sequence $\{\gamma_{k}\}_{k=1}^{\infty}$.
\end{proof}

Now we present some preliminary estimates that will be needed in proving that the set $P_{\infty}$ has large enough Hausdorff content.

The next lemma follows from the Poincar\'e inequality and it is our key tool for the proof of Theorem \ref{maintheorem}. 
It transforms the information of the Poincar\'e inequality from integrals to estimates on Hausdorff content of level sets.
We use it to estimate the number of points in the sets $P_{k}$ that were constructed in the previous section.
For the proof see Theorem 5.9 in \cite{hk98}. The case $p=1$ and $s=Q-1$ follows from combining the arguments from the proofs of Theorem 5.9 in \cite{hk98} and Theorem 3.6 in \cite{kkst}.

\begin{lemma}\label{lem:hk}
Suppose that $0\leq Q-p<s\leq Q-1$, or $p=1$ and $s=Q-1$,
and $E,F\subset B(r)$ are compact. If
\[
\min\{\mathcal H^{s}_{\infty}(E),\mathcal H^{s}_{\infty}(F)\}\geq \lambda r^{s}
\]
for some $0<\lambda\leq 1$, then for any Lipschitz function $f$ such that $f=1$ in $E$ and $f=0$ in $F$, we have
\[
\int_{B( r)}\!(\mathrm{Lip} f)^p\,d\mu\geq  \frac{1}{C}\lambda r^{Q-p},
\]
where $C=C(s,p,Q,c_A,c_P)\geq 1$.
\end{lemma}

The next lemma enables us to estimate the number of balls in the families $P_{k}$.
\begin{lemma}
\label{lem:numberofballs}
Fix $0\leq Q-p<s\leq Q-1$, or $p=1$ and $s=Q-1$.
Let $\Omega\subset X$ be a domain, $\omega\in\partial\Omega$, $0<r<\diam(\Omega)$, $B=B(\omega,r)$, and $\eta>0$ sufficiently small (say $\eta<\eta_1<1$, which will be specified in the proof). Assume that 
\begin{equation}\label{eqn:bigboundary}
\mathcal{H}_{\infty}^{s}\big( B \cap \partial\Omega \big)\geq {C}_0 r^{s} 
\end{equation}
holds 
for some constant ${C}_0>{0}$. Suppose that there exists $z_{B}\in\partial B\cap \Omega$ such that $B(z_{B},r/2)\subset \Omega$. Let $\{B_i(\eta r)\}_{i=1}^{N}$ be a maximal collection of balls inside $B\cap\Omega$ that is well placed along $\partial\Omega$ and such that $\{B_i(\eta r)\}_{i=1}^{N}\cup \{B(z_{B},\eta r)\}$ is a subset of a finitely chainable collection of balls. Then 
$$
N\geq \frac{1}{K}\frac{1}{\eta^{Q-p}},
$$ 
where $K=K({C}_0,s,p,Q,c_A,c_P)$.
\end{lemma}

\begin{proof}
Let $F= \overline{B(z_{B},r/4)\cap B}$ and $E=\partial \Omega\cap B$, which are compact subsets of $B$ as $X$ is proper. Let $\mathcal B$ be a maximal chainable collection of balls of radius $\eta r$ in $\Omega$ such that 
\begin{enumerate}
\item[(i)] $\{B_i(\eta r)\}_{i=1}^{N} \cup \{B(z_{B},\eta r)\}\subset \mathcal B$,
\item[(ii)]  $x\in  B$ for each $B(x,\eta r)\in \mathcal B$, and 
\item[(iii)] $d(x,y)\geq \eta r/2$ whenever $x\neq y$ and $B(x,\eta r), B(y,\eta r)\in \mathcal B$. 
\end{enumerate}
As the distance of the centres of the balls in $\mathcal B$ is bounded by $\eta r/2$ from below, the number of the balls in $\mathcal B$ is bounded by some constant $S=S(\eta,c_{A},Q)$.

Now define
\[
f(x)=\begin{cases}
1, & x\in  B\setminus D\\
g(x)=\max\limits_{1\leq i\leq N}\left[1-\frac{d(x,40B_{i})}{\eta r}\right]_{+}, & x\in   B\cap D.
\end{cases}
\]
Here
\[
D= \bigcup_{B\in\mathcal B}20B,
\] 
i.e. it consists of all points that are close to (or in) the maximal chainable set.
Now $f=1$ in $E$ and  $f=0$ in $F$. The first claim is clear if $x\in E\setminus D$. Let us consider the case $x\in E \cap D$. This means that $x\in B(y,20\eta r)$ with some $B(y,\eta r)\in\mathcal B$. By considering balls of radius $\eta r$ that have centres on the geodesic connecting $y$ to $x$, we find a point $\widetilde y$ such that $d(x,\widetilde y)\leq d(x,y)$, $B(\tilde y,\eta r)$ is in $\Omega$ and touches the boundary of $\Omega$. As $\{B_i(\eta r)\}_{i=1}^{N}$ is a maximal collection of balls inside $B\cap\Omega$ that is well placed along $\partial\Omega$, there exists $i\in 1,2,\ldots N$ such that $B(\tilde y,4\eta r)\cap 4B_i\neq \emptyset$ as, otherwise, we could add $B(\tilde y,\eta r)$ to the collection. It follows that the distance of $x$ to the centre of $B_i$ is at most $28\eta r$ and consequently $f(x)=1$.

The claim that $f=0$ on $F$ holds clearly whenever the balls $41B_i$ do not touch $F$. As each ball $B_i$ touches the boundary of $\Omega$ and $B(z_B,r/2)\subset \Omega$, this holds at least when $42\eta r<r/2-r/4$ i.e. if $\eta<\eta_1\leq 164$. 

Moreover, $f$ is a $\tfrac{1}{\eta r}$--Lipschitz function in $B$. To see this, notice first that $g$ is $\tfrac{1}{\eta r}$--Lipschitz as a maximum of $\tfrac{1}{\eta r}$--Lipschitz functions. As $f$ is constant outside $D$, we only need to check what happens at $\partial  D\cap B$. Notice that the balls $10B_{i}, i=1,2,\ldots,N$ cover the subset of $\{x\in \Omega\cap B\,:\, d(x,\partial\Omega)=\eta r\}$ that belong to the union of the balls in $\mathcal B$ as, otherwise, we could add more balls to the collection of balls that are well placed along $\partial \Omega$. If $x\in \partial D \cap (\Omega\cap  B)$ then $x\in 20B_{i}$ for some $i\in\{1,2,\ldots,N\}$ and consequently $g(x)=1$.

The local Lipschitz constant of $f $ satisfies
\begin{eqnarray*}
\mathrm{Lip}f(x)=0& \text{if} & x\notin \bigcup_{i=1}^{N}41 B_i \\ 
0\leq \mathrm{Lip}f(x)\leq \frac{1}{\eta r}& \text{if} & x\in \bigcup_{i=1}^{N}41 B_i.
\end{eqnarray*}
This shows that $\mathrm{Lip}f$ is integrable on $  B$. Thus, 
using that $\mu$ is $Q$--Ahlfors regular,
$$
\int_{B}\!(\mathrm{Lip}f)^p\,{d}\mu\leq \frac{1}{(\eta r)^p}\mu\left( \bigcup_{i=1}^{N}41 B_i  \right) \leq \frac{1}{(\eta r)^p}\sum_{i=1}^{N}c_{A} (41\eta r)^Q= Nc_{A}41^Q(\eta r)^{Q-p}.
$$
Lemma \ref{lem:hk} implies that $\int_{ B}( \mathrm{Lip}f)^p d\mu\geq Cr^{Q-p}$. Combining this with the previous estimate, we obtain

\begin{equation}\label{eqn:N}
N\geq\frac{1}{K}\frac{1}{\eta^{Q-p}},
\end{equation}
where $K$ only depends on the data related to the space and $C_{0}$, i.e. $K=K(C_{0},s,p,Q,c_A,c_P)$.

\end{proof}

\begin{lemma}
\label{lem:1}
Let $P_\infty$ be as in the construction and
$\eta_1, K$ be as in Lemma \ref{lem:numberofballs}. Suppose that \eqref{eqn:bigboundary} is satisfied with some $s$ for all $\omega\in\partial\Omega$ and $0<r<r_{0}$. Fix $0<\eps<s$ and $0<\eta<\eta_2=\min(\eta_1,K^{-2/\eps})$. 
Then $\mathcal{H}^{s-\eps}_{\infty}(P_\infty) \geq Cr_{0}^{s-\eps}$,
where  $C=C(C_{0},s,p,Q,c_A,c_P,\eps,\eta)>0$. 
\end{lemma}

\begin{proof}
Fix $0<\eps<s$. Let $p=Q-s+\tfrac{\eps}{2}$ if $s<Q-1$ and $p=1$ if $s=Q-1$. Assume without loss of generality that $r_0=d_{\Omega}(z_0)=1$.

Let us construct a sequence of probability measures $\{\nu_k\}_{k=0}^\infty$ as follows. First, we set
\[
\nu_0=\delta_{\omega_0},
\]
where $P_0=\{\omega_0\}$, using the notation built up in Section \ref{section:construction}.
For $k\geq1$, we define the measures as follows. For every $\omega\in P_{k-1}$, let $P_\omega\subset P_k$ denote the $k$th generation descendants of $\omega$ and $N_\omega$ the number of points in $P_\omega$. Then we define $\nu_k$ to be the measure supported on $P_k$ that satisfies
\[
\nu_k=\sum_{\omega\in P_{k-1}}\frac{1}{N_\omega}\nu_{k-1}(\omega)\delta_{P_\omega}.
\]

One can show that the sequence $\{\nu_k\}_{k=0}^\infty$ converges weakly to some limiting probability measure $\nu$, in the sense that
$$
\lim_{k\rightarrow\infty}\int_{X}\!f\,{d}\nu_k=\int_{X}\!f\,{d}\nu
$$
for every bounded continuous function $f$ on $X$. Moreover, $\nu$ is supported on $P_\infty$.

In order to prove the lemma, it suffices to show that the measure $\nu$ satisfies $\nu\big(B(r)\big)\lesssim r^{s-\eps}$ for every ball $B(r)$ that intersects $P_\infty$. Then, for any cover $\{B(r_i)\}$ of $P_\infty$ it would follow that
\[
1=\nu(P_\infty)\leq \sum_{i}\nu\big(B(r_i) \big)\lesssim \sum_{i}r_{i}^{s-\eps}.
\]
Hence, taking an infimum over all covers of $P_\infty$ by balls,
\[
\mathcal{H}^{s-\eps}_{\infty}(P_\infty) \gtrsim 1,
\]
as desired.

Write $N_k$ for the number of points in $P_k$. From the construction, $N_0=1$, and, by Lemma \ref{lem:numberofballs}, since $\eta<\eta_1$, $N_{k}\geq K^{-1}\eta^{-(Q-p)}N_{k-1}$ for all $k\geq{1}$. From here it follows that
\[
\nu_k(2B_j)\leq K\eta^{Q-p}\nu_{k-1}(2B_{j-1}),
\]
where $B_j\in\mathcal{B}_j$, $B_{j-1}\in\mathcal{B}_{j-1}$, and $2B_j\subset 2B_{j-1}$. Applying this $j-1$ more times yields
\[
\nu(2B_j)\leq K^{j}\eta^{j(Q-p)}\leq\eta^{j(Q-p-\eps/2)} \leq \eta^{j(s-\eps)},
\]
as $\eta$ satisfies $0<\eta<K^{{-2/\eps}}$. 

Fix a ball $B\subset{X}$ that intersects $P_\infty$. If its radius $r$ satisfies $r\geq{1}$, then 
$$
\nu(B)\leq 1\leq r^{s-\eps}.
$$
If $r<1$, however, consider $j$ such that $\eta^{j+1}\leq r <\eta^{j}$. Then $B$ intersects at most one ball $2B_j$, where $B_j\in\mathcal{B}_j$. It follows that for $0<\eta<K^{-2/\eps}$, 
$$
\nu(B)\leq \nu(2B_j)\leq \eta^{j(s-\varepsilon)}\leq \eta^{-s+\eps}r^{s-\eps},
$$
completing the proof.
\end{proof}

\begin{proof}[Proof of Main Theorem]
Fix $0<\eps<s$, $z_0\in\Omega$, $r=r_0=d_\Omega(z_0)$, and $0<\eta<\eta_2$. It suffices to show that $P_\infty\subset\partial\Omega_{z_0}(c) \cap \partial \Omega$ for some $c\geq 1$ as then, by Lemma \ref{lem:1}, we would have 
$$
\mathcal{H}^{s-\eps}_{\infty}( \partial\Omega_{z_0}(c) \cap \partial \Omega  )\geq \mathcal{H}^{s-\eps}_{\infty}(P_\infty)\geq Cr_0^{s-\eps},
$$ 
where $C$ is as in the lemma.

Suppose that $\omega\in P_k$ for some $k$. That is, there exists a $B_k=B(z_k,\eta^kr)\in\tilde{\mathcal{B}}_k$ such that $\omega\in \partial B_k\cap\partial\Omega$ for some $z_k\in\Omega$. Since $X$ is geodesic and $B_k\subset\Omega$, we can connect $\omega$ to $z_k$ by a $1$--John curve, $\gamma_{k+1}$. By construction, $\tilde{\mathcal{B}}_k$ satisfies the conditions of Lemma \ref{lem:2b} and so there exists a path $\gamma_{k}$ connecting $z_k$ to $z_{k-1}$, where $z_{k-1}$ is the centre of some ball $B_{k-1}\in\tilde{\mathcal{B}}_{k-1}$, such that
$$
d_{\cup\tilde{\mathcal{B}}_k}(\gamma_k)\geq\frac{\eta^{k}r}{2} \qquad\text{and}\qquad\ell(\gamma_{k})\lesssim \eta^{k} r.
$$ 
Applying Lemma \ref{lem:2b} again $k-1$ more times yields a sequence of paths $\{\gamma_j\}_{j=1}^{k}$ such that $\gamma=\gamma_1\cup\gamma_2\cup\ldots\cup\gamma_k$ connects $z_k$ to $z_0$ and such that for each $j=1,2,\ldots,k$,
$$
d_{\cup\tilde{\mathcal{B}}_j}(\gamma_j)\geq\frac{\eta^jr}{2} \qquad\text{and}\qquad\ell(\gamma_j)\lesssim \eta^j r.
$$
Thus, for some $M>1$, we have that
$$
d_{\Omega}(\gamma_j)>\frac{1}{M}\eta^jr\qquad\text{and}\qquad\ell(\gamma_j)\leq M \eta^j r
$$
for each $j=1,2,\ldots,k$. Therefore, applying Lemma \ref{lem:4John}, it follows that $\gamma$ is $c$--John for some $c$. The path $\tilde{\gamma}=\gamma\cup\gamma_{k+1}$ connecting $\omega$ to $z_0$ can then be shown to be $(1+c)$--John, demonstrating that $P_k$ is a subset of the visible boundary.  
As the obtained constant $c$ is independent of $k$, this approach shows that $\bigcup_{k}P_k$ is a subset of the visible boundary. In fact, since the visible boundary is closed, this implies that $P_\infty$ is a subset of the visible boundary, and this completes the proof.
\end{proof}

\begin{remark}\label{remark}
The construction of Theorem \ref{maintheorem} also works in doubling metric measure spaces that are not $Q$--Ahlfors regular. The doubling condition (together with the connectedness of the space) implies that there exist some constants $0<Q_{2}\leq Q_{1}<\infty$ such that
\[
\frac{1}{C} \left(\frac{r}{R}\right)^{Q_{1}}\leq\frac{\mu(B(y,r))}{\mu(B(x,R))}\leq C\left(\frac{r}{R}\right)^{Q_{2}}
\]
for all $x\in X$, $0<r\leq R<\infty$ and $y\in B(x,R)$.
In the doubling setting, it is more efficient to work with Hausdorff content of certain codimension. 
Our method of estimating the size of visible boundary works in principle with the doubling measure, but we seem to lose $Q_1-Q_2$ in the size of exponents.
\end{remark}

\section{Necessity of Assumptions}

We now give an example that demonstrates that it is necessary to assume that the boundary is thick at all locations and at all scales. In the first example (Theorem \ref{thm:example1}), the complement of the domain can be chosen to be a closure of a connected open set. In our second example (Theorem \ref{thm:example2}), we see that with disconnected complement, the visual boundary can be made very small.

\begin{theorem}\label{thm:example1}
Let $n\geq{3}$. Fix $c\geq{1}$, $0\leq\eps<n-1$, and $0<\eta<\tfrac14$. Then there exists a domain $\Omega\subset \R^n$ with connected boundary, and a $z_0\in\Omega$ such that, for $r=d_{\Omega}(z_0)$,
$$
\mathcal{H}_{\infty}^{n-1}\big( B(z_0,2r) \cap \partial\Omega \big)\geq r^{n-1}, 
$$
but
$$
\mathcal{H}^{n-1-\eps}_{\infty}( \partial\Omega_{z_0}(c) \cap \partial \Omega  ) \leq \eta r^{n-1-\eps}.
$$
\end{theorem}

\begin{proof}
Let $z_0=0$ and $\{p_i\}_{i=1}^{N}$ be a set of points on $\partial B\left(0,\frac{3}{2}\right)$ such that 
$$
\partial B(0,3/2)\subset\bigcup_{i=1}^{N}B(p_i,\tfrac{1}{2c}).
$$
For each $i=1,2,\ldots,N$, choose a connected and closed set $A_i\subset\R^n\setminus B(0,\frac{3}{2})$ containing $p_i$ such that $A_i\cap\partial B(0,2)\neq\emptyset$, $\operatorname{diam}(A_i)\leq 1$ and
$$
\mathcal{H}^{n-1-\eps}_{\infty}(A_i)\leq\frac{\eta}{N}.
$$
As $\operatorname{diam}(A_i)\leq 1$, this implies that 
$$\mathcal{H}^{n-1}_{\infty}(A_i)\leq\mathcal{H}^{n-1-\eps}_{\infty}(A_i)\leq\frac{\eta}{N}.
$$
Now, let
$$
\Omega=B(0,2)\setminus \bigcup_{i=1}^{N}A_i.
$$
From this, we see that $r=d_{\Omega}(z_0)=\frac{3}{2}$. 

Consider any path $\gamma$ connecting the origin to some point $\omega_0\in \partial\Omega\cap\partial B(0,2)$. Writing $z$ for the point in $i(\gamma)\cap\partial B(0,\frac{3}{2})$, we have that $\ell(\gamma(z,\omega_0))\geq \frac{1}{2}$ but $cd_{\Omega}(z)\leq\frac{1}{4}$ and so no points of $\partial B(0,2)$ are part of the visible boundary of $\Omega$ near $z_0$. 
Hence,
$$
\mathcal{H}^{n-1-\eps}_{\infty}(\partial\Omega_{z_0}(c) \cap \partial \Omega)\leq\sum_{i=1}^{N}\mathcal{H}^{n-1-\eps}_{\infty}(A_i)\leq \eta \leq \eta r^{n-1-\eps}
$$
but
\[
\mathcal{H}^{n-1}_{\infty}\big(B(0,2r)\cap\partial\Omega \big)=\mathcal{H}^{n-1}_{\infty}(\partial\Omega) \geq\mathcal{H}^{n-1}_{\infty}\big(\partial B(0,2)\setminus \bigcup_{i=1}^{N}A_i\big)\geq (1-\eta)2^{n-1}\geq r^{n-1}.
\]
\end{proof}

\begin{theorem}\label{thm:example2}
Let $n\geq{2}$ and fix $0\leq\eps<n-1$. There exists a domain $\Omega\subset \R^n$ and a $z_0\in\Omega$ such that, for $r=d_{\Omega}(z_0)$,
$$
\mathcal{H}_{\infty}^{n-1}\big( B(z_0,2r) \cap \partial\Omega \big)\geq r^{n-1}, 
$$
but
$$
\mathcal{H}^{n-1-\eps}_{\infty}( \partial\Omega_{z_0}(c) \cap \partial \Omega  )=0
$$
for any $c\geq 1$.
\end{theorem}

\begin{proof}
Let $z_0=0$ and fix $c\geq{1}$. Select a set of points $\{p_i\}_{i=1}^{N}$ on $\partial B(0,2-2^{1-c})$ such that 
$$
\partial B(0,2-2^{1-c})\subset\bigcup_{i=1}^{N}B(p_i,\tfrac{1}{c2^c}).
$$
Let
$$
\Omega=B(0,2)\setminus \bigcup_{i=1}^{N}p_i.
$$
From this, we see that $r=d_{\Omega}(z_0)=2-2^{1-c}$. 

Consider any path $\gamma$ connecting the origin to some point $\omega_0\in \partial B(0,2)$. Writing $z$ for a point in $i(\gamma)\cap\partial B(0,2-2^{1-c})$, we have that $\ell(\gamma(z,\omega_0))\geq 2^{1-c}$ but $cd_{\Omega}(z)\leq 2^{-c}$, implying that $\partial B(0,2)\cap \partial\Omega_{z_0}(c)=\emptyset$. Therefore,
$$
\mathcal{H}^{n-1-\eps}_{\infty}(\partial\Omega_{z_0}(c) \cap \partial \Omega)\leq\sum_{i=1}^{N}\mathcal{H}^{n-1-\eps}_{\infty}(p_i)=0
$$
but
\[
\mathcal{H}^{n-1}_{\infty}\big(B(z_0,2r)\cap\partial\Omega \big) \geq\mathcal{H}^{n-1}_{\infty}(\partial B(0,2))= 2^{n-1}\geq r^{n-1},
\]
where the inequality holds since $\partial\Omega$ contains $\partial B(0,2)$ (up to a finite set of points, if $c=1$). Recall that a set of a finite number of points has zero $s$--dimensional Hausdorff content so long as $s\neq 0$.

\end{proof}

\section{Acknowledgements}

The authors would like to thank the referees for their useful comments, corrections, and recommended improvements. The first author was partially supported by the Fonds de recherche du Qu\'{e}bec -- Nature et technologies (FRQNT). Part of the work for this project was done while the first author was visiting Aalto University; he would like to thank that institution for their kind hospitality. The second author was partially supported by Academy of Finland, project 308063.


\vskip .5cm

\noindent Author Information:

\vskip .3cm

\noindent Ryan Gibara

\noindent Address: Department of Mathematical Sciences, P.O. Box 210025, University of
Cincinnati, Cincinnati, OH 45221--0025, U.S.A.

\noindent Email: {\tt ryan.gibara@gmail.com}

\vskip .3cm

\noindent Riikka Korte

\noindent Address: Department of Mathematics and Systems Analysis, Aalto University, P.O.~Box~11100, FI-00076 Aalto, Finland.

\noindent Email: {\tt riikka.korte@aalto.fi}

\end{document}